\documentclass[12pt]{amsart}
\usepackage{paralist}
\usepackage{amssymb}
\usepackage{amstext}
\usepackage{amsmath}
\usepackage{amscd}
\usepackage{latexsym}
\usepackage{amsfonts}
\usepackage{color}

\theoremstyle{plain}
\newtheorem{thm}{Theorem}[section]
\newtheorem*{thm*}{Theorem}
\newtheorem*{cor*}{Corollary}

\newtheorem{prop}[thm]{Proposition}
\newtheorem{lem}[thm]{Lemma}
\newtheorem{cor}[thm]{Corollary}

\newtheorem*{claim*}{Claim}

\theoremstyle{definition}
\newtheorem{defn}[thm]{Definition}
\newtheorem{ex}[thm]{Example}

\theoremstyle{remark}

\numberwithin{equation}{thm}

\def\a{\mathfrak a}

\def\e{\mathrm{e}}
\def\m{\mathfrak m}

\def\p{\mathfrak p}
\def\q{\mathfrak q}

\def\N{\Bbb N}
\def\Z{\Bbb Z}

\newcommand{\ov}{\overline}

\def\depth{\mathrm{depth}}
\def\Supp{\mathrm{Supp}}
\def\Ann{\mathrm{Ann}}
\def\Ass{\mathrm{Ass}}

\def\height{\mathrm{ht}}

\def\Spec{\mathrm{Spec}}

\def\R{{\mathcal R}}

\newcommand{\mapright}[1]{%
\smash{\mathop{%
\hbox to 1cm{\rightarrowfill}}\limits^{#1}}}

\newcommand{\mapleft}[1]{%
\smash{\mathop{%
\hbox to 1cm{\leftarrowfill}}\limits_{#1}}}

\begin{document}

 \setlength{\baselineskip}{14pt}
\title[Sally module and the first normal Hilbert coefficient]{A filtration of the Sally module and the First normal Hilbert Coefficient}

\author[Masuti]{Shreedevi K. Masuti}
\address{Dipartimento di Matematica, Universit{\`a} di Genova, Via Dodecaneso 35, 16146 Genova, Italy}
\email{masuti@dima.unige.it, masuti.shree@gmail.com}

\author[Ozeki]{Kazuho Ozeki}
\address{Department of Mathematical Sciences, Faculty of Science, Yamaguchi University, 1677-1 Yoshida, Yamaguchi 753-8512, Japan}
\email{ozeki@yamaguchi-u.ac.jp}

\author[Rossi]{Maria Evelina Rossi}
\address{Dipartimento di Matematica, Universit{\`a} di Genova, Via Dodecaneso 35, 16146 Genova, Italy}
\email{rossim@dima.unige.it}

\date{\today}
\thanks{SKM was supported by INdAM COFOUND Fellowships cofounded by Marie Curie actions, Italy. 
KO was partially supported by Grant-in-Aid for Scientific Researches (C) in Japan (15K04820). MER was partially supported by PRIN 2015EYPTSB-008 Geometry of Algebraic Varieties.}
\maketitle
\begin{abstract}
The Sally module of an ideal is an important tool to interplay between Hilbert coefficients and the properties of the associated graded ring. In this paper we give new insights on the structure of the Sally module. We apply these results characterizing 
the almost minimal value of the first Hilbert coefficient in the case of the normal filtration in an analytically unramified Cohen-Macaulay local ring.  
\end{abstract}

{\footnotesize
}


\section{Introduction}
The Hilbert coefficients of a primary ideal in a Noetherian local ring are  important numerical invariants associated to an ideal.   Among them,   the multiplicity and the first Hilbert coefficient (also called Chern number) have been extensively studied,   due to their geometric meaning.   Even if the Hilbert coefficients  give asymptotic information, very often they determine the properties of the ring itself. For instance, if the ring is Cohen-Macaulay, then it is regular if and only if the multiplicity of the maximal ideal is one. Also, a conjecture stated by W. V. Vasconcelos and proved recently  by several authors (see  \cite{GGHOPV}) says that, if the ring is unmixed, then it is Cohen-Macaulay once the first Hilbert coefficient vanishes for some parameter ideal. 

  It happens that  blow-up algebras associated to an ideal  have good  homological properties if  the corresponding Hilbert coefficients achieve ``extremal" values with respect to some bounds.  This is one of the reasons for which there is a lot of interest in finding   sharp bounds on the Hilbert coefficients. This is also the line  traced  by   C. Huneke   and   J. Sally  in their  work (see for instance \cite{H87}, \cite{HH99}, \cite{Sa77}). For a general  overview  see \cite{RV10}. 
\vskip 2mm

Because of geometric reasons,  normal Hilbert coefficients 
are even more significant. Let $(R,\m)$ be an analytically unramified Cohen-Macaulay local ring of dimension $d>0$ with infinite residue field $R/\m$ and $I$ an $\m$-primary ideal of $R$.
Let $\overline{I}$ denote the integral closure of $I$. Consider the so called {\it{normal filtration }}  $\{\overline{I^n}\}_{n \in \Z} $ and we are interested in the corresponding  Hilbert-Samuel polynomial.  
It is well-known that there are integers $\ov{\e}_i(I)$, called   the {\it normal Hilbert coefficients} of $I, $ such that for $n \gg 0$
\begin{eqnarray*}
 \ell_R(R/\ov{I^{n+1}})=\ov{\e}_0(I) \binom{n+d}{d}-\ov{\e}_1(I) \binom{n+d-1}{d-1}+\cdots+(-1)^d \ov{\e}_d(I).
\end{eqnarray*}
 Here $\ell_R(N)$ denotes, for an $R$-module $N$, the length of $N$.
Since $R/\m$ is infinite there exists a minimal reduction $J=(a_1,\ldots,a_d)$ of $I $ and, under our assumptions,  there exists a positive  integer $r$\ such that $\ov{I^{n+1}}=J\ov{I^n}$ for $n \ge r.$ 

\noindent We set $\ov{\rm{r}}_J(I):=\min \{r \geq 0 \ | \ \mbox{$\ov{I^{n+1}}=J\ov{I^n}$ for all $n \geq r$} \}$ the reduction number of $I.$
\vskip 2mm

We recall that $ \ov{\e}_1(I) \geq 0, $ but the bound can be more precise.  It is well known that 
$$
\ov{\e}_1(I) \geq \ov{\e}_0(I)-\ell_R(R/\ov{I}).
$$
 Moreover, the equality  holds true if and only if $\ov{I^{n+1}}=J^n\ov{I}$ for every $n \geq 0$ and for every minimal reduction $J$ of $I$ (that is $\ov{\rm{r}}_J(I) \leq 1$) $($\cite{N60, H87, Itoh92, O87}$)$. In this case the normal associated graded ring $\ov{G}(I):=\oplus_{n \geq 0} \ov{I^n}/\ov{I^{n+1}}$ of $I$ is Cohen-Macaulay.
Recently \cite{CPR16} showed that if the equality $\ov{\e}_1(I) = \ov{\e}_0(I)-\ell_R(R/\ov{I})+1$ holds true, then $\depth ~\ov{G}(I) \geq d-1$ $($see also Corollary \ref{CPR}$)$.
This equality was explored by Phuong \cite{Phu15} in the case $R$ is generalized Cohen-Macaulay. 
\vskip 2mm

By \cite{EV91, GR98, Itoh92, HM97} it is known that  
$$ \ov{\e}_1(I) \geq \ov{\e}_0(I) -\ell_R(R/\ov{I}) + \ell_R(\ov{I^2}/J\ov{I})$$
 and the equality  holds true if and only if $\ov{I^{n+1}}=J^{n-1}\ov{I^2}$ for every $n \geq 1$ (that is $\ov{\rm{r}}_J(I) \leq 2$). 
In this case the normal associated graded ring $\ov{G}(I)$ of $I$ is Cohen-Macaulay $($see also Corollary \ref{EV}$)$. We notice   that     $\ell_R(\ov{I^2}/J\ov{I})$ does not depend on a minimal reduction $J$ of $I  $  (see for instance \cite{RV10}). 

Thus the ideals $I$ with $\ov{\e}_1(I) = \ov{\e}_0(I)-\ell_R(R/\ov{I})+ \ell_R(\ov{I^2}/J\ov{I})$ enjoy nice properties and it seems natural to ask when the equality  $\ov{\e}_1(I)=\ov{\e}_0(I)-\ell_R(R/\ov{I})+\ell_R(\ov{I^2}/J \ov{I})+1$ holds. The main purpose of this paper is to explore this equality which is the content of Section 3. We present in this case the structure of the Sally module and, in particular, we prove that   $\depth ~\ov{G}(I) \geq d-1$ (Theorem \ref{maintheorem}). We remark that if $I$ is integrally closed, the corresponding equality was studied in \cite{OR16}. In \cite{OR16} the authors proved that if the equality $\e_1(I)=\e_0(I)-\ell_R(R/I)+\ell_R(I^2/JI)+1$ holds true, then  the depth of the associated graded ring $G(I):=\bigoplus_{n\geq 0} I^n/I^{n+1}$  can be any integer between $0$ and $d-1$ (see \cite[Theorem 5.1]{OR16}). This ``bad" behavior motives our study in the case of the normal filtration proving that it  enjoys nice properties as compared to the $I$-adic filtration. 

\vskip 2mm

As the title outlines, an important tool in this  paper is  the Sally module introduced by  W. V. Vasconcelos  \cite{V94}.   The aim of the author   was to define a module in between the associated graded ring and the Rees algebra taking care of important information coming from a minimal reduction.   Actually,   a more detailed information comes from the graded parts of a suitable filtration $\{C^{(i)} \}$ of the Sally module that was introduced by M. Vaz Pinto in \cite{VP96}.   In Section 2  we prove   some important results  on $C^{(2)} $  which will be key ingredients  for proving the main result. Some of them are stated in a very general setting. Our hope is that these will be successfully applied to give new insights  to   problems  related to the normal Hilbert coefficients,    for instance \cite{Itoh92}.     In the last section we deduce  some   consequences of Theorem \ref{maintheorem} which include  the already quoted result of A. Corso, C. Polini and M. E. Rossi \cite[Theorem 2.5]{CPR16}. 

\section{Filtering the Sally module}
In this section we study the Sally module associated to any $I$-admissible filtration $\mathcal{I}. $
Following M. Vaz Pinto \cite{VP96} we introduce a filtration of the Sally module, $C^{(l)}(\mathcal{I})$ for $l \geq 1.$ This approach is extremely useful for relating the properties of the Hilbert coefficients and the  graded module associated to  an $I$-admissible filtration $\mathcal{I}, $ as evidenced in \cite{CPR16}, \cite{OR16}.  In \cite{CPR16} the authors analyzed the Sally module (=$C^{(1)}(\mathcal{I})$) of the normal  filtration to study the equality $\bar{\e}_1(I)=\bar{\e}_0(I)-\ell_R(R/\ov{I})+1.$ In order to investigate the equality $\bar{\e}_1(I)=\bar{\e}_0(I)-\ell_R(R/\ov{I})+\ell_R(\ov{I^2}/J \ov{I})+1, $ in this section we prove  
some important properties of $C^{(2)}(\mathcal{I}). $  These properties will play an important role in proving our main result in the next section. 

\vskip 2mm
We recall that $C^{(2)}(\mathcal{I})$ has been studied in \cite{OR16} for $I$-adic filtration. For our purpose we need more deep results.   
\vskip 2mm

Throughout this section, let $(R,\m)$ be a Cohen-Macaulay local ring (not necessarily analytically unramified) and $I$ an $\m$-primary ideal in $R.$ Recall that a {\it a filtration} of ideals $\mathcal{I}:=\{I_n\}_{n \in \mathbb{Z}}$ is a chain of ideals in $R$ such that $R=I_0$ and $I_n \supseteq I_{n+1}$ for all $n \in \mathbb{Z}$.
We say that a filtration $\mathcal{I}$ is {\it $I$-admissible} if for all $m,n \in \mathbb{Z},$ $I_m \cdot I_n \subseteq I_{m+n}, ~I^n \subseteq I_n$ and there exists $k \in \mathbb{N}$ such that $I_{n} \subseteq I^{n-k}$ for all $n \in \mathbb{Z}.$ It is well known that  if $R$ is analytically unramified, then $\{\ov{I^n}\}_{n \in \mathbb{Z}}$ is an $I$-admissible filtration.
\vskip 2mm

For an $I$-admissible filtration $\mathcal I=\{I_n\}_{n \in \mathbb{N}},$ let 
$$
\R(\mathcal I)=\sum_{i\geq 0}I^it^i \subseteq R[t], \ \ \ \R'(\mathcal I)=\sum_{i \in \mathbb{Z}}I^it^i \subseteq R[t,t^{-1}], \ \ \mbox{and} \ \ G(\mathcal I)=\R'(\mathcal I)/t^{-1}\R'(\mathcal I) 
$$
denote, respectively, the Rees algebra, the extended Rees algebra, and the associated graded ring of $\mathcal{I}$ where $t$ is an indeterminate over $R$. 
We set
$$\ov{\mathcal{R}}(I) := \sum_{n \geq 0}\ov{I^n}t^n  \ \subseteq R[t], \ \ ~~~~\ov{\mathcal{R'}}(I) := \sum_{n \in \mathbb{Z}}\ov{I^n}t^n \ \subseteq R[t,t^{-1}], \ \
\operatorname{ and } 
\ov{G}(I):= \ov{\mathcal{R'}}(I)/t^{-1}\ov{\mathcal{R'}}(I) 
$$
for the Rees algebra, the extended Rees algebra and the associated graded ring of $\{\ov{I^n}\}_{n \in \Z},$ respectively.

Since $R/\m$ is infinite there exists a minimal reduction $J=(a_1,a_2,\ldots,a_d)$ of $\mathcal{I},$ that is there exists an integer $r \in \mathbb{Z}$ such that the equality $I_{n+1}=JI_n$ holds true for all $n \geq r$.
Let 
\[
\rm{r}_J(\mathcal I):=\min\{r \geq 0 \ | \ \mbox{$I_{n+1}=JI_n$ hold true for all $n \geq $ r } \}
\]
be the {\it reduction number} of $\mathcal{I}$ with respect to $J$.
We set 
$$ T:=\mathcal{R}(J):=\mathcal{R}(\{J^n\}_{n \in \mathbb{Z}})$$
and $\mathcal{M}=\m T+T_+$ denotes the graded maximal ideal of $T$. Then $\mathcal{R}(\mathcal{I})$ is a module finite extension of $T$. Hence there exist integers $\e_i(\mathcal{I})$, called as the {\it Hilbert coefficients} of $\mathcal{I}$ such that the equality
\begin{eqnarray*}
 \ell_R(R/{I_{n+1}})=\e_0(\mathcal{I}) \binom{n+d}{d}-\e_1(\mathcal{I}) \binom{n+d-1}{d-1}+\cdots+(-1)^d \e_d(\mathcal{I})
\end{eqnarray*}
holds true for all $n \gg 0$ $($c.f. \cite[Proposition 3.1]{Mar89}$)$.
This polynomial is called the {\it Hilbert-Samuel polynomial} of $\mathcal{I}$.
For a graded $T$-module $E$ and $\alpha \in \Z$ we denote by $E(\alpha)$ the graded $T$-module whose  grading is given by $[E(\alpha)]_n=E_{\alpha+n}$ for all $n \in \Z$.
\vskip 2mm

Following Vasconcelos \cite{V94}, we consider 
$$
S_J(\mathcal{I}):= \frac{\mathcal{R}(\mathcal{I})_{\geq 1} t^{-1}}{I_1 T}\cong \bigoplus_{n \geq 1}I_{n+1}/J^nI_1
$$ 
the {\it Sally module} of $\mathcal{I}$ with respect to $J$. Notice that $S_J(\mathcal{I})$ is a finite $T$-module. In \cite{VP96} Vaz Pinto introduced a filtration of the Sally module in the case $\mathcal{I}=\{I^n\}_{n \in \Z}$. Following this line,  we extend the definition to  any $I$-admissible filtration $\mathcal{I}$.

\begin{defn}\label{VP}
For each $l \geq 1$, consider the $T$-module
$$
C^{(\ell)}_J(\mathcal{I}):= \frac{\mathcal{R}(\mathcal{I})_{\geq \ell} t^{-1}}{I_{\ell} Tt^{\ell-1}} \cong \bigoplus_{n \geq \ell} {I_{n+1}}/J^{n-\ell+1}I_{\ell}. 
$$
Let $L^{(\ell)}_J(\mathcal{I})= [C^{(\ell)}(\mathcal{I})]_{\ell} T$ be the $T$-submodule of $C^{(\ell)}(\mathcal{I}).$
Then
\[
L^{(\ell)}_J(\mathcal{I}) \cong \bigoplus_{n \geq \ell} J^{n-\ell}I_{\ell+1}/J^{n-\ell+1}I_{\ell}.
\] 
Hence for every $\ell \ge 1$  we have the following natural exact sequence of graded $T$-modules
$$ 0 \to L^{(\ell)}_J(\mathcal{I}) \to C^{(\ell)}_J(\mathcal{I}) \to C^{(\ell+1)}_J(\mathcal{I}) \to 0. $$
\end{defn}
Throughout this section we set 
\[
 S:=S_J(\mathcal{I}),~ C^{(\ell)}:=C^{(\ell)}_J(\mathcal{I}) \mbox{ and }~L^{(\ell)}:=L^{(\ell)}_J(\mathcal{I}) 
\]
unless otherwise specified. Notice that $C^{(1)}=S,$ and since $\mathcal{R}(\mathcal I)$ is a finite graded $T$-module, $C^{(\ell)}$ and $L^{(\ell)}$ are finitely generated graded $T$-modules for every $\ell \geq 1.$ Moreover, if $\rm{r}=\rm{r}_J(\mathcal{I})$, then $C^{(\rm{r})}=0.$ 

Let us begin with the following lemma.

\begin{lem}\label{fact1}
Let $\ell \geq 1$ be an integer.
Then the following assertions hold true.
\begin{enumerate}
\item \label{fact(1)} $\m^{k} C^{(\ell)} = (0)$ for integers $k \gg 0$; hence ${\dim}_T~C^{(\ell)} \leq d$.
\item \label{fact(2)} $C^{(\ell)}=(0)$ if and only if ${\mathrm{r}}_J(\mathcal{I}) \leq \ell$.
\end{enumerate}
\end{lem}
\begin{proof}
\begin{asparaenum}[(1)]
\item  Since
\[
 [C^{(\ell)}]_n = \begin{cases}
                   0 & \mbox{ if } n \leq \ell-1 \\
                   I_{n+1}/J^{n-\ell+1}I_{\ell} & \mbox{ if } n \geq \ell,
                  \end{cases}
\]
for each $n \geq \ell$ there exists an integer $k_n$ such that $\m^{k_n} [C^{(\ell)}]_n=0.$ Since $C^{(\ell)}$ is a finite $T$-module, there exists an integer $k'$ such that $\m^{k'} C^{(\ell)} = (0).$ Hence the first assertion follows. Therefore 
\[
 \dim_T~ C^{(\ell)}=\dim_T ~ T/\Ann_T(C^{(\ell)}) \leq \dim~ T/\m T=d. 
\]

\item This follows immediately from the definition of $C^{(\ell)}.$ \qedhere
\end{asparaenum}
\end{proof}

In the one dimensional case, we have the following elementary, but useful lemma.

\begin{lem}\label{d=1}
Suppose that $d=1.$ Then $C^{(\ell)}$ is a Cohen-Macaulay $T$-module for any $I$-admissible filtration $\mathcal{I}$ and for all $ 1 \leq  \ell \leq \rm{r}-1$ where $\rm{r}=\rm{r}_J(\mathcal{I}).$ 
\end{lem}
\begin{proof}
Let $J=(a).$ 
We claim that $at \in [T]_1$ is a non-zero divisor of $C^{(\ell)}.$ 
Let $n \geq \ell$ and $x \in I_{n+1}$ be such that $x a \in J^{n-\ell+2} I_{\ell}.$ 
Then $x a = a^{n-\ell+2} x'$ for some $x' \in I_{\ell}. $ 
Since $a$ is $R$-regular, $x=a^{n-\ell+1} x' \in J^{n-\ell+1} I_{\ell} .$ Therefore $xt^n=0 $ in $C^{(l)}.$ This proves the lemma.
\end{proof}

In this paper, the structure of the graded module $C^{(2)}$ plays an important role. We derive some basic properties of $C^{(2)}$ which we need.

In the following result we need that $J \cap I_2=JI_1$ holds true.
This condition is automatically satisfied if $\mathcal{I}=\{\m^n\}_{n \in \Z}$ or if $\mathcal{I}=\{\overline{I^n}\}_{n \in \mathbb{Z}}$ (see \cite{H87, Itoh88, HU14}).
We also notice that
$$\ell_R(I_2/JI_1)=\e_0(\mathcal{I})+(d-1)\ell_R(R/I_1)-\ell_R(I_1/I_2)$$ 
holds true (see for instance \cite[Corollary 2.1]{RV10}), so that  $\ell_R(I_2/JI_1)$ does not depend on a minimal reduction $J$ of $\mathcal{I}$. We remark that the following proposition (\eqref{C_2(1)}-\eqref{C_2(4)}) was proved in \cite[Propositions 2.2, 2.8, and 2.9]{OR16} in the case $\mathcal{I}=\{I^n\}_{n \in \Z}, $ but the proof we present is different. 
 Moreover, Proposition \ref{C_2}\eqref{C_2(5)} improves \cite[Lemma 2.11]{OR16}. 

Throughout this paper $\mathrm{HS}_{M}(z)$ denotes the Hilbert series of a graded module $M.$ 

\begin{prop}\label{C_2}
Let $\p=\m T$ and suppose that $J \cap I_2=JI_1$. Then the following assertions hold true.
\begin{enumerate}
\item \label{C_2(1)} $\Ass_T~(C^{(2)}) \subseteq \{\p\}$. Hence $\dim_TC^{(2)}=d$, if $C^{(2)} \neq (0)$.
\item For all $n \geq 0,$
\begin{eqnarray*}
\ell_R(R/I_{n+1})&=&\e_0(\mathcal{I})\binom{n+d}{d}-\{\e_0(\mathcal{I})-\ell_R(R/I_1)+\ell_R(I_2/JI_1)\}\binom{n+d-1}{d-1}\\
&& + \ell_R(I_2/JI_1)\binom{n+d-2}{d-2}-\ell_R([C^{(2)}]_n).
\end{eqnarray*}
\item \label{C_2(3)} $\e_1(\mathcal{I})=\e_0(\mathcal{I})-\ell_R(R/I_1)+\ell_R(I_2/JI_1)+\ell_{T_{\p}}(C^{(2)}_{\p}).$
\item  \label{C_2(4)} $\mathrm{HS}_{G(\mathcal{I})}(z)=\frac{\ell_R(R/I_1)+(\e_0(\mathcal{I})-\ell_R(R/I_1)-\ell_R(I_2/JI_1))z+\ell_R(I_2/JI_1)z^2}{(1-z)^d}-(1-z)\mathrm{HS_{C^{(2)}}}(z)$.
\item \label{C_2(5)} Suppose $C^{(2)} \neq (0)$ and let $c=\depth_TC^{(2)}$. 
Then $\depth ~G(\mathcal{I}) =c-1$, if $c < d$. Moreover, $\depth~ G(\mathcal{I}) \geq d-1$ if and only if $C^{(2)}$ is a Cohen-Macaulay $T$-module.
\end{enumerate}
\end{prop}

In the proof of Proposition \ref{C_2} we need the following lemmata. 
We set $\mathcal{R}'(I)=\mathcal{R}'(\{I^n\}_{n \in \Z})$ and $G(I)=G(\{I^n\}_{n \in \Z})$.

\begin{lem} \label{basic}
Suppose $J \cap I_2=J I_1$ where $J=(a_1,a_2,\ldots,a_d).$ Then  
 \begin{enumerate}
  \item \label{basic1} for all $n \geq 0,$ $J^{n+1} \cap J^n I_2=J^{n+1} I_1.$ 
  \item \label{basic2} $(a_1,a_2,\ldots,a_i) \cap J^{n+1} I_2=(a_1,a_2\ldots,a_i) J^n I_2$ for all $i=1,2,\ldots,d$ and $n \geq 0. $   In particular, $a_i t$ is a non-zero divisor on $T/I_2 T$ for every $1\leq i \leq d.$ 
 \end{enumerate}
\end{lem}
\begin{proof}
$(1)$ It suffices to show that $J^{n+1} \cap J^n I_2 \subseteq J^{n+1} I_1$. 
Let $x \in J^{n+1} \cap J^n I_2.$ Write $x=\sum_{\alpha \in \N^d,~|\alpha|=n} x_{\alpha} a_1^{\alpha_1}a_2^{\alpha_2} \ldots a_d^{\alpha_d}$ where $x_\alpha \in I_2.$ Since $x \in J^{n+1}$ 
and $a_1,\ldots,a_d$ is a regular sequence in $R,$ 
$x_\alpha \in J $ for every $\alpha.$ Thus $x_\alpha \in J \cap I_2=JI_1$ and hence $x \in J^{n+1} I_1$ as required.

$(2)$ We prove the assertion using decreasing induction on $i.$ 
The assertion is clear for $i=d.$ 
Assume that $i<d$ and the assertion is true for $i+1.$ 
It is enough to prove that 
$$(a_1, a_2, \ldots, a_i) \cap J^{n+1} I_2 \subseteq (a_1, a_2, \ldots,a_i) J^n I_2.$$ We use induction on $n.$ 
Suppose $n=0$ and 
$$x \in (a_1, a_2, \ldots,a_i) \cap J I_2 \subseteq (a_1, a_2, \ldots, a_i, a_{i+1}) \cap J I_2.$$ 
Hence $x \in (a_1, a_2, \ldots,a_i,a_{i+1}) I_2$ by induction on $i$.
Let $x=x'+ a_{i+1}x_{i+1}$ where $x' \in (a_1,a_2,\ldots,a_i)I_2$ and $x_{i+1} \in I_2$.
Thus $a_{i+1}x_{i+1} \in (a_1,a_2,\ldots,a_i)$ which implies that 
$$
x_{i+1} \in (a_1,a_2,\ldots,a_i) \cap I_2= ((a_1,a_2,\ldots,a_i) \cap J) \cap I_2=(a_1,a_2,\ldots,a_i) \cap  J I_1.
$$ 
Since $a_1,a_2,\ldots,a_{d}$ forms a regular sequence in $R$, the map $\phi: (R/I_1)^d \to J/JI_1$ defined as 
\[
\phi(\overline{r_1},\overline{r_2},\ldots,\overline{r_d})=\overline{a_1r_1+a_2r_2+\cdots+a_dr_d}
\] 
is an isomorphism. Hence 
\[
 (a_1,a_2,\ldots,a_i) \cap  J I_1 = (a_1,a_2,\ldots,a_i) I_1.
\]
Thus $x \in (a_1,a_2,\ldots,a_i) I_2.$ 

Assume that $n \geq 1$ and that our assertion holds true for $n-1$.
Let $x \in (a_1,a_2,\ldots,a_i) \cap J^{n+1} I_2$.
We then have 
$$ 
(a_1,a_2,\ldots,a_i) \cap J^{n+1} I_2 \subseteq (a_1,a_2,\ldots,a_i,a_{i+1}) \cap J^{n+1} I_2=(a_1,a_2,\ldots,a_{i+1}) J^{n} I_2
$$  by induction on $i$.
Write $x=x' + a_{i+1} x_{i+1}$ where $x' \in (a_1,a_2,\ldots,a_i)J^{n}I_2$ and $x_{i+1} \in J^n I_2.$ 
Then $a_{i+1}x_{i+1} \in (a_1,a_2,\ldots,a_i)$ which implies that 
$$x_{i+1} \in (a_1,a_2,\ldots,a_i) \cap J^n I_2=(a_1,a_2,\ldots,a_i) J^{n-1} I_2$$ by induction on $n.$ 
Hence $x \in (a_1,a_2,\ldots,a_i) J^n I_2$ as required. \qedhere
\end{proof}

\begin{lem}\label{exact} 
Suppose $J \cap I_2=J I_1$ where $J=(a_1, a_2, \ldots,a_d).$ 
Then for every $1\leq i \leq d$ the sequences of $T$-modules
\begin{align}
 \label{Eqn:aiexact} 0 \longrightarrow T/I_1 T \overset{a_i}{\longrightarrow}& T/I_2 T \longrightarrow T/[I_2T+a_iT]  \longrightarrow 0 \\
 \label{Eqn:aitexact} 0 \longrightarrow (T/[I_2 T+JT])(-1) \overset{a_it}{\longrightarrow}& T/[I_2T+a_iT] \longrightarrow  T/[I_2 T+a_it T+a_iT] \longrightarrow 0
\end{align}
are exact. 
\end{lem}
\begin{proof}
To prove the exactness of \eqref{Eqn:aiexact} it suffices to prove that $\mu_{a_i}:T/I_1 T \overset{a_i}{\to} T/I_2 T$ is injective. Consider $xt^n \in T_n$ where $x \in J^n$ and assume that $\mu_{a_i}(\overline{xt^n})=\overline{a_ixt^n}=0$  in $T/I_2 T.$ 
Here $\overline{(\cdot)}$ denotes the image of an element in respective quotients. Thus $a_i x \in J^{n+1} \cap J^nI_2=J^{n+1} I_1$ by Lemma \ref{basic}\eqref{basic1}. 
Since 
\[
 T/I_1 T \simeq (R/I_1)[X_1,\ldots,x_d],
\]
$a_it$ is a non-zero divisor on $T/I_1 T.$ Therefore $x \in J^n I_1$ and hence $\mu_{a_i}$ is injective.
 
 Now we prove that the sequence \eqref{Eqn:aitexact} is exact. It is enough to prove that the map 
 \[
 \mu_{a_it}:(T/[I_2 T+JT])(-1) \overset{a_it}{\to} T/[I_2T+a_iT]
 \]
 is injective. 
 Let $xt^n \in T_n$ and assume that $\mu_{a_it} (\overline{xt^n})=\overline{a_ixt^{n+1}}=0$ in $T/[I_2T+a_iT]$.
Thus $a_i x \in (I_2+a_i)J^{n+1}$.
Let $a_ix=y+a_iz$ where $y \in I_2J^{n+1}$ and $z \in J^{n+1}.$ 
Then $a_i(x-z) \in (a_i) \cap J^{n+1}I_2=a_i J^n I_2$ by Lemma \ref{basic}\eqref{basic2}. 
Thus $x-z \in J^n I_2$ and hence $x \in (I_2+J)J^{n}$. Therefore $\overline{xt^n} =0$ in $T/[I_2 T+JT]$ as required.
\end{proof}

The following proposition is crucial in the proof of Proposition \ref{C_2}.

\begin{prop}\label{CMness}
Suppose $J \cap I_2=J I_1.$ Then $\left({T}/{I_2 T}\right)_P$ is Cohen-Macaulay $T_P$-module for all $\mathcal M \neq P \in \Supp_{T} \left({T}/{I_2 T}\right).$
\end{prop}
\begin{proof}
Let $P \in \Supp_{T} \left({T}/{I_2 T}\right)$ and $P \neq \mathcal M.$ Then $\m \subseteq P$ and $T_{1} \nsubseteq P.$ This implies that $a_it \notin P$ for some $i.$ Hence $[I_2 T+a_it T+a_iT]_P=T_P.$ Therefore by \eqref{Eqn:aitexact} 
\[
  ((T/[I_2 T+JT])(-1))_P  \overset{a_it}{\longrightarrow} (T/[I_2T+a_iT])_P 
\]
is an isomorphism. As 
\[T/[I_2 T+JT]\simeq R/(I_2+ J) \otimes_R G(J)= ({R}/(I_2+J))[X_1,\ldots,X_d], \]
$ ((T/[I_2 T+JT])(-1))_P$ is Cohen-Macaulay $T_P$-module. Hence $(T/[I_2T+a_iT])_P$ is Cohen-Macaulay $T_P$-module. Also, 
\[T/I_1 T\simeq R/I_1 \otimes_R G(J)= ({R}/{I_1})[X_1,\ldots,X_d]\]
implies that $(T/I_1 T)_P$ is Cohen-Macaulay $T_P$-module. By \eqref{Eqn:aiexact} the sequence
\[
0 \longrightarrow (T/I_1 T)_P \overset{a_i}{\longrightarrow} (T/I_2 T)_P \longrightarrow (T/[I_2T+a_iT])_P \longrightarrow 0
\]
is exact. Since $(T/I_1 T)_P$ and $(T/[I_2T+a_iT])_P$ are Cohen-Macaulay $T_P$-module of same dimension, 
$\left({T}/{I_2 T}\right)_P$ is Cohen-Macaulay $T_P$-module. \qedhere
\end{proof}

The following techniques are due to M. Vaz Pinto \cite[Section 2]{VP96}. For $\ell \geq 1$ let 
$$
D^{(\ell)}:=D_J^{(\ell)}(\mathcal I):=(I_{\ell+1}/ J I_\ell) \otimes_R (T/\Ann_R(I_{\ell+1}/JI_{\ell})T)
$$ 
be the graded $T$-module. Notice that 
$$
T/\Ann_R(I_{\ell+1}/JI_{\ell})T \cong (R/\Ann_R(I_{\ell+1}/JI_{\ell}))[X_1,X_2,\cdots,X_d]
$$ 
is the polynomial ring with $d$ indeterminates over the ring $R/\Ann_R(I_{\ell+1}/JI_{\ell})$. 
Hence
\[
 D^{(\ell)} \cong (I_{\ell+1}/J I_{\ell})[X_1,\ldots,X_d].
\]
Let
$$ \theta_\ell: D^{(\ell)}(-\ell) \to L^{(\ell)}$$
be an epimorphism of graded $T$-modules defined as the identity on $I_{\ell+1}/ J I_{\ell}$ and that sends   each generator $X_i$ to a generator of $L^{(\ell)}$ as a $T$-module. More precisely, $\theta_\ell(\ov{c_i} X_i)=\ov{c_i a_i t^{\ell+1}} $ where $c_i \in I_{\ell+1}.$ Extending by linearity, $\theta_{\ell}$ gives a graded $T$-module homomorphism which is surjective.

The following lemma is also a key for our proof of Proposition \ref{C_2}.

\begin{lem}\label{L1}
Let $\ell \geq 1$ and suppose that $J \cap I_{\ell+1}=JI_{\ell}$.
Then the map $\theta_{\ell}: D^{(\ell)}(-\ell) \to L^{(\ell)}$ is an isomorphism of graded $T$-modules.
\end{lem}
\begin{proof}
It suffices to show that $K_\ell:=\ker \theta_\ell=0.$ Let $n \geq l$ and $F(X_1,\ldots,X_d) \in [K_\ell]_n \subseteq [D^{(\ell)}(-\ell)]_n.$ Then 
$F(X_1,\ldots,X_d)$ is a homogeneous polynomial of degree $n-\ell $ with coefficients in $I_{\ell+1}/ J I_{\ell}.$ 
Write $F(X_1,\ldots,X_d)=\sum_{\alpha \in \N^d,|\alpha|=n-\ell} \ov{b_\alpha} X_1^{\alpha_1} \cdots X_d^{\alpha_d}$ where $b_{\alpha} \in I_{\ell+1}.$
Since $\theta_1(F(X_1,\ldots,X_d))=0,$ $\sum_{\alpha \in \N^d,|\alpha|=n-\ell} {b_\alpha} a_1^{\alpha_1} \cdots a_d^{\alpha_d} \in J^{n-\ell+1} I_\ell \subseteq J^{n-\ell+1}. $ As $a_1,\ldots,a_d$ forms a regular sequence in $R,$ $b_\alpha \in J$ for every $\alpha.$ Thus $b_\alpha \in J \cap I_{\ell+1}= J I_\ell$ and hence $\ov{b_\alpha}=0.$ Therefore $F=0$ as required.
\end{proof}

We are now in a position to prove Proposition \ref{C_2}.

\begin{proof}[Proof of Proposition \ref{C_2}]
\begin{asparaenum}[(1)]
\item 
Suppose $C^{(2)} \neq (0).$ Let $P \in \Ass_T~(C^{(2)})$.
Then by Lemma \ref{fact1}\eqref{fact(1)} $\m T \subseteq P.$
Suppose that $\m T \subsetneq P.$ Then $\height_TP \geq 2$, because $\m T$ is a height one prime ideal in $T$.
We look at the following exact sequences of $T_P$-modules
$$ 0 \to I_2 T_P \to (\mathcal{R}(\mathcal{I})_{\geq 2})_P \to C^{(2)}_P \to 0 \ \ \ (*_3) \ \ \ 
\mbox{and} \ \ \ 0 \to I_2 T_P \to T_P \to  T_P/I_2 T_P \to 0 \ \ \ (*_4) $$
which follow from the canonical exact sequences
$$ 0 \to I_2 T(-1) \to (\mathcal{R}(\mathcal{I})_{\geq 2})(1) \to C^{(2)} \to 0 \ \ \ \mbox{and} \ \ \ 0 \to I_2 T \to  T \to T/I_2 T \to 0 $$ of $T$-modules.
We notice here that $\depth_{T_P} (\mathcal{R}(\mathcal{I})_{\geq 2}(1))_P > 0 $, because $a_1 \in P $ is a non-zero divisor on $\mathcal{R}(\mathcal{I})_{\geq 2}(1)$.
Thanks to the depth lemma and the exact sequence $(*_3)$, we get $\depth_{T_P}I_2 T_P=1$, because $\depth_{T_P}C^{(2)}_P=0$. Since $T$ is Cohen-Macaulay, $\depth_{T_P} T_P \geq 2$. Hence from the exact sequence $(*_4)$ we get that $\depth_{T_P}T_P/I_2 T_P=0.$ 
On the other hand, $\depth_{T_P}T_P/I_2T_P=\height~ P-1>0 $ by Proposition \ref{CMness} if $P \neq \mathcal{M}$, and $\depth_{T_{\mathcal{M}}}T_{\mathcal{M}}/I_2T_{\mathcal{M}}>0$ by Lemma \ref{basic}\eqref{basic2}. However, it is contradiction.
Thus, $P=\m T$ as required.

\item
For all $n \geq 0$, let us look at the following two exact sequences
$$0 \to I_{n+1}/J^nI_1 \to R/J^nI_1 \to R/I_{n+1} \to 0 \ \ \mbox{and} \ \ 0 \to J^n/J^nI_1 \to R/J^nI_1 \to R/J^n \to 0$$
of $R$-modules.
Then, because $S_n \cong I_{n+1}/J^nI_1$, $J^n/J^nI_1 \cong (R/I_1)^{\binom{n+d-1}{d-1}}$, and $\ell_R(R/J^n)=\e_0(\mathcal I)\binom{n+d-1}{d}$, we get
\begin{eqnarray*}
\ell_R(R/I_{n+1})&=&\e_0(\mathcal{I})\binom{n+d}{d}-\{\e_0(\mathcal{I})-\ell_R(R/I_1)\}\binom{n+d-1}{d-1}-\ell_R(S_n),
\end{eqnarray*}
for all $n \geq 0$ (c.f. \cite[Proposition 2.2(2)]{GNO08} and \cite[Equation (6.3)]{RV10}).
Then thanks to the exact sequence
$$ \hspace*{1cm} 0 \longrightarrow L^{(1)} \longrightarrow S \longrightarrow C^{(2)} \longrightarrow 0 \ \ \ \hspace*{2cm}(\dagger_1) $$
and isomorphism $L^{(1)} \cong D^{(1)}(-1) \cong ((I_2/J I_1)[X_1,\ldots,X_d])(-1)$ by Lemma \ref{L1}, we have 
$$ 
\ell_R(S_n)=\ell_R(I_2/JI_1)\binom{n+d-2}{d-1}+\ell_R([C^{(2)}]_n) .
$$
Thus we get the required equality.

\item Using assertion $(2)$ and the argument as in \cite[Proposition 2.2(3)]{GNO08} we get 
$\e_1(\mathcal{I})=\e_0(\mathcal{I})-\ell_R(R/I_1)+\ell_{T_\p}(S_{\p}).$ Now using the exact sequence 
$(\dagger_1)$ and Lemma \ref{L1} we get the desired equality. 

\item
We have $$\mathrm{HS}_{G(\mathcal{I})}(z)=\frac{\ell_R(R/I_1)+(\e_0(\mathcal{I})-\ell_R(R/I_1))z}{(1-z)^d}-(1-z){\mathrm{HS}_{S}}(z)$$
(c.f. \cite[Proposition 6.3]{RV10}) and
$${\mathrm{HS}_{S}}(z)=\frac{\ell_R(I_2/JI_1)z}{(1-z)^d}+{\mathrm{HS}_{C^{(2)}}}(z)$$
by the above exact sequence $(\dagger_1)$ and isomorphism $L^{(1)} \cong D^{(1)}(-1)$ of graded $T$-modules by Lemma \ref{L1}.
Therefore, we get the required equality.

\item
Thanks to the above exact sequence $(\dagger_1)$ and isomorphism $L^{(1)} \cong D^{(1)}(-1)$ of graded $T$-modules by Lemma \ref{L1}, we have $\depth_T S \geq c$.
Hence $\depth~ G(\mathcal{I}) \geq c-1$ because $\depth ~G(\mathcal{I}) \geq \depth_T S-1$ by \cite[Proposition 6.3]{RV10}.

In the rest of the proof of assertion (5), it suffices to show that 
$$\min\{d-1, \depth~G(\mathcal{I})\} \leq c-1.$$
We proceed by induction on $l=\depth~ G({\mathcal{I}})$.
Thanks to Proposition \ref{C_2}\eqref{C_2(1)} we have $c>0$ so that we may assume that $l \geq 1$ and that our assertion holds true for $l-1$.

We choose $a_1t$ which is both $G(\mathcal{I})$-regular, $C^{(2)}$-regular (as $\Ass_T (C^{(2)})=\{\m T\}$) and $J=(a_1,a_2,\ldots,a_d).$ Let $$\alpha : R[t] \to (R/(a_1))[t]$$ be the natural $R$-algebra map defined by $\alpha(t)=t$.
We set $\mathcal{I}'=\{[I_n+(a_1)]/(a_1)\}_{n \in \mathbb{Z}}$ and $J'=J/(a_1)$.
Then, since $\alpha([\mathcal{R}(\mathcal{I})]_n)=[\mathcal{R}(\mathcal{I}')]_{n} \cong [I_n+(a_1)]/(a_1)$ and $\alpha(T_n)=[\mathcal{R}(J')]_n \cong [J^n+(a_1)]/(a_1)$ for all $n \geq 0$, the map $\alpha$ induces surjective homomorphisms $$\mbox{$\mathcal{R}(\mathcal{I}) \to \mathcal{R}(\mathcal{I}')$ and $T \to \mathcal{R}(J')$}$$ of graded $R$-algebras.
Thanks to these homomorphisms, we get the natural epimorphism 
$$\phi:C^{(2)} \to C^{(2)}(\mathcal{I}').$$
Then since $\alpha(a_1t)=0$ we have $\phi(a_1t \cdot C^{(2)})=(0)$, 
so that we have the epimorphism 
$$
\overline{\phi}:C^{(2)}/(a_1t C^{(2)}) \to C^{(2)}(\mathcal{I}')
$$ 
of graded $T$-modules.
Because $a_1t$ is $G(\mathcal{I})$-regular, the map $\overline{\phi}$ is an isomorphism so that $\depth_{R(J')}C_2(\mathcal{I}')=c-1<d-1$.
We also have $J'\cap{\mathcal I'}_2=J{\mathcal I'}_1$ and as $a_1t$ is $G(\mathcal {I})$-regular, $\depth~ G(\mathcal{I'}) =l-1$.
Hence by the hypothesis of induction on $l$, we get
$$
\min\{\dim R/(a_1)-1, \depth~ G(\mathcal{I}')\} \leq \depth_{R(J')}C_2(\mathcal{I}')-1=c-2.
$$
Then since $\dim R/(a_1)=d-1$ and $\depth~ G(\mathcal{I'}) =l-1$, we get $\min\{d-1, l\} \leq c-1$ as required. 
This completes the proof of assertion (5). \qedhere
\end{asparaenum}
\end{proof}

Combining Proposition \ref{C_2}\eqref{C_2(1)}, \ref{C_2}\eqref{C_2(3)}, \ref{C_2}\eqref{C_2(4)} and Lemma \ref{fact1}\eqref{fact(2)}, and using the Valabrega-Valla criterion (c.f.\cite[Theorem 1.1]{RV10}) we obtain the following result that was proven by Elias and Valla \cite[Theorem 2.1]{EV91} in the case $\mathcal{I}=\{\m^n\}_{n \in \mathbb{N}}$ and by Guerrieri and Rossi \cite[Theorem 2.2 and Proof of Proposition 2.3]{GR98} for any $I$-admissible filtration.

\begin{cor}\label{EV}
Suppose that $J \cap I_2=JI_1$.
Then we have 
$$\e_1(\mathcal{I}) \geq \e_0(\mathcal{I})-\ell_R(R/I_1)+\ell_R(I_2/JI_1).$$
The equality $\e_1(\mathcal{I}) = \e_0(\mathcal{I})-\ell_R(R/I_1)+\ell_R(I_2/JI_1)$ holds true if and only if $\rm{r}_J(\mathcal{I}) \leq 2$.
When this is the case, the following assertions hold true:
\begin{itemize}
\item[$(i)$] $\mathrm{HS}_{G(\mathcal{I})}(z)=\frac{\ell_R(R/I_1)+(\e_0(\mathcal{I})-\ell_R(R/I_1)-\ell_R(I_2/JI_1))z+\ell_R(I_2/JI_1)z^2}{(1-z)^d},$
\item [$(ii)$] if $d \geq 2,$ then $\e_2(\mathcal{I})=\ell_R(I_2/J I_1)=\e_1(\mathcal{I})-\e_0(\mathcal{I})+\ell_R(R/I_1)$ and $\e_i(\mathcal{I})=0$ for all $3 \leq i \leq d$, and 
\item[$(iii)$] $G(\mathcal{I})$ is a Cohen-Macaulay ring.
\end{itemize}
\end{cor}

We now prove an important property of $C^{(2)}$ in Proposition \ref{ChasSn} which plays a crucial role in the proof of the main result.

\begin{prop}\label{ChasSn}
Let $d \geq 3$ and $0 \leq n \leq d-1.$ Suppose that $J \cap I_2=JI_1$ and $\R(\mathcal I)$ satisfies   Serre's property $(S_n)$ as a $T$-module. Then $C^{(2)}$ also satisfies   Serre's property $(S_n)$ (as $T$-module).
\end{prop}
\begin{proof}
Consider $\mathcal M \neq P \in \Supp_T C^{(2)}.$ 
Let $h=\dim_{T_P} C^{(2)}_P$ and $r=\min\{h,n\}.$ Consider the exact sequence
 \[
  0 \to (I_2 T)_P \to T_P \to (T/I_2T)_P \to 0. 
  \]
 By Proposition \ref{CMness} $(T/I_2T)_P$ is Cohen-Macaulay $T_P$-module of dimension $h$. Since $T_P$ is Cohen-Macaulay of dimension $h+1,$ we get that $\depth_{ T_P} (I_2T)_P \geq h+1 \geq r+1.$  Since $I_1 t \nsubseteq  P ,$ 
 $(\R(\mathcal{I})/(\R(\mathcal I)_{\geq 2})_P =0.$ This implies that $(\R(\mathcal{I}))_P \simeq (\R(\mathcal I)_{\geq 2})_P. $ In particular, $\depth_{T_P} (\R(\mathcal I)_{\geq 2})_P \geq r. $ 
 Hence 
from the exact sequence 
 \[
  0 \to (I_2 T(-1))_P \to (\R(\mathcal I)_{\geq 2} (1))_P \to (C^{(2)})_P \to 0
 \]
we get that $\depth_{T_P}  (C^{(2)})_P \geq r.$

Now, let $P=\mathcal M.$ Since $\depth_T \R(\mathcal{I}) \geq n, $ $\depth_T G(\mathcal{I}) \geq n-1.$ 
This implies that $\depth_T S \geq n.$ Since $L^{(1)} \simeq (I_2/J I_1)[X_1,\ldots,X_d](-1),$ $(L^{(1)})_\mathcal{M}$ is Cohen-Macaulay $T_\mathcal{M}$-module. Hence from the exact sequence
\[
 0 \to (L^{(1)})_{\mathcal{M}} \to {S}_{\mathcal{M}} \to (C^{(2)})_{\mathcal{M}} \to 0,
\]
we get that $\depth_{T_\mathcal{M}} (C^{(2)})_{\mathcal{M}} \geq n.$ \qedhere
\end{proof}

\section{Main Theorem}
In this section we prove the main result of this paper.
In the rest of this paper, we assume that $(R,\m)$ is an analytically unramified Cohen-Macaulay local ring.
We set $\overline{C}:=C^{(2)}_J(\{\overline{I^n}\}_{n \in \Z})$ and $B:=T/\m T \cong (R/\m)[X_1,X_2,\cdots,X_d]$ the polynomial ring with $d$ indeterminates over the field $R/\m$. 

The main result of this paper is stated as follows. 

\begin{thm}\label{maintheorem}
Let $(R,\m)$ be an analytically unramified Cohen-Macaulay local ring of dimension $d>0$ and $I$ an $\m$-primary ideal in $R.$ Then following statements are equivalent:
\begin{enumerate}
\item[(1)]  $\overline{\e}_1(I)=\ov{\e}_0(I)-\ell_R(R/\ov{I})+\ell_R(\ov{I^2}/J \ov{I})+1;$
\item[(2)] $\ov{C} \simeq B(-m)$ as graded $T$-modules for some $m \geq 2;$
\item[(3)] $\ell_R(\ov{I^{m+1}}/J \ov{I^m})=1$ and $\ov{I^{n+1}}=J \ov{I^n}$ for all $2 \leq n \leq m-1$ and $n \geq m+1$ for some $m \geq 2$.
\end{enumerate}
In this case, the following assertions follow:
\begin{itemize}
\item[(i)] ${\ov{\rm{r}}}_J(I)=m+1$.
\item[(ii)] $HS_{\ov{G}(I)}(z)=\frac{\ell_R(R/I)+(\e_0(I)-\ell_R(R/\ov{I})-\ell_R(\ov{I^2}/J \ov{I}))z+\ell_R(\ov{I^2}/J \ov{I})z^2-z^m+z^{m+1}}{(1-z)^d}.$ 
\item[(iii)] $\ov{\e}_2(I)=\ell_R(\ov{I^2}/J \ov{I})+m$ and $\ov{\e}_i(I)=\binom{m}{i-1}$ for $3 \leq i \leq d.$
\item[(iv)] $\depth ~\ov{G}(I) \geq d-1.$
\item[(v)] $\ov{G}(I)$ is Cohen-Macaulay if and only if $\ov{I^3} \nsubseteq J.$ In this case, we have $m=2.$ 
\end{itemize}
\end{thm}
\begin{proof}
$(1) \Rightarrow (2):$
By Proposition \ref{C_2}\eqref{C_2(3)} we have $\ell_{T_\p}(\ov{C}_\p)=1$ where $\p=\m T$. 
Hence $\ov{C} \neq (0)$ and $(\p \ov{C})_\p=0$. 
Because $\Ass_T (\p \ov{C}) \subseteq \Ass_{T} (\ov{C}) =\{\p\}$ by Proposition \ref{C_2}\eqref{C_2(1)}, we get that $\p \ov{C}=0$.
In particular, we have $\m \ov{C}=0$ so that $\ov{C}$ is a graded $B$-module. 
Let $Q(B)$ denote the quotient field of $B.$ 
Since $\ell_{T_\p} (\ov{C}_\p)=1,$ we conclude that 
$$
\ov{C}_\p=\ov{C}_{\p B} \simeq B_\p \otimes_B \ov{C}=Q(B) \otimes_B \ov{C} \simeq Q(B). $$ 
As $\Ass_T(\ov{C})=\{\p\},$ the canonical $B$-homomorphism $\ov{C} \rightarrow \ov{C}_\p=Q(B)$ is injective. 
This gives an inclusion $\ov{C} \hookrightarrow B$ because $\ov{C}$ is a finite $B$-module.  
Therefore there exists a graded ideal $\a \subseteq B$ and an integer $m$ such that $\ov{C} \cong \mathfrak{a}(-m)$ as graded $T$-modules.
We claim that $\mathfrak{a} =B$.

Suppose $d \leq 2.$ Then $\ov{C}$ is a Cohen-Macaulay $T$-module by Proposition \ref{C_2}\eqref{C_2(5)}, because $\depth ~\overline{G}(I) \geq 1$. Hence $\ov{C}$ is Cohen-Macaulay as a $B$-module.
Therefore, $\ov{C}$ is $B$-free so that $\mathfrak{a}\cong B$.

Suppose that $d \geq 3$ and that $\a \subsetneq B.$ 
We may assume that $\height_B ~\a \geq 2.$ 
Indeed suppose that $\height_B ~\a \leq 1.$ 
Then because $\ov{C}\neq 0,$ $\height_B~ \a=1.$
Let $\{(f_1),\ldots,(f_k)\}$ be the set of all prime ideals in $B$ such that $\a \subseteq (f_i), ~i=1,\ldots,k.$ Then $\a \subseteq (f_1) \cap \ldots \cap (f_k)=(f_1 \ldots f_k). $ Write $\a=f_1^{n_1} \ldots f_k^{n_k} \a'$ for some ideal $\a'$ in $B$ and some positive integers $n_1,\ldots,n_k$ so that $\a' \nsubseteq (f_i)$ for all $i=1,\ldots,k.$ Then $\a \simeq \a'.$ If $\a'=B$ then we are done. Otherwise $\height_B ~\a' \geq 2.$ Hence we assume that $\height_B~ \a \geq 2. $ Let $\q \in \Spec~ T$ be such that $\q B \in \Ass_B ((B/\a)(-m)). $ Then $\height_B ~\q B \geq 2.$ 
By \cite[Proposition 3.1]{Phu15} $\ov{\R}(I)$ has the Serre's property $(S_2)$ (and so has the Serre's property $(S_2)$ as a $T$-module by \cite[Claim 3.5]{Phu15}). 
Hence by Proposition \ref{ChasSn}, $\ov{C}$ has the Serre's property $(S_2)$ as a $T$-module. 
Therefore $\depth_{B_\q} \ov{C}_\q=\depth_{T_\q} \ov{C}_\q \geq 2.$  Hence from the exact sequence
\[
 0 \to \ov{C}_{\q} \to B(-m)_{\q} \to ((B/\a)(-m))_{\q}\to0
\]
we get that $\depth_{B_\q} ((B/\a)(-m))_{\q} >0$ (as $B_\q$ is Cohen-Macaulay). This contradicts to the fact that $\q B_\q \in \Ass_{B_\q} ((B/\a)(-m)).$ 
Hence $\a =B.$ This proves the claim.

Moreover, since $[\ov{C}]_{m}=B(-m)_{m}=B_0 \neq 0$ and $\ov{C}_n=0$ for $n \leq 1,$ we get that $m \geq 2.$
\vskip 2mm 

$ (2) \Rightarrow (3)$: 
Since $\ov{C}_n=(0)$ for $2\leq n<m,$ $\ov{I^{n+1}}=J^{n-1} \ov{I^2}$ for $2 \leq n<m.$ Hence $\ov{I^{n+1}}=J \ov{I^n}$ for $2 \leq n <m.$ 
We have \[\ell_R(\ov{I^{m+1}}/J \ov{I^m})=\ell_R(\ov{C}_m)=\ell_R(B_0)=1.\]
Now, let $n \geq m+1$ and $x \in \ov{I^{n+1}}.$ Since $\ov{C}_n=B_{n-m}\ov{C}_m, $ there are $g_i \in J^{n-m}$ and $x_i \in \ov{I^{m+1}}$ such that $ x = g_1 x_1 + \cdots+ g_k x_k $ modulo $J^{n-1}\ov{I^2}.$ Hence $x\in J \ov{I^n}.$ Thus $\ov{I^{n+1}}=J \ov{I^n}$ for all $n \geq m+1.$ 
 
$(3) \Rightarrow (1)$:
By assumption, $\ov{C}_n=(0)$ for all $n \leq m-1$, $\ov{C}_{n+m}=B_n\ov{C}_m$ for all $n\geq 0$, and $\ell_R(\overline{C}_m)=\ell_R(\overline{I^{m+1}}/J\overline{I^{m}})=1$.
Hence there exists a surjective map
\[
 \phi:B(-m) \to \ov{C} \to 0.
\]
Let $K=\ker \phi.$ Since $e(B(-m))=1$ and $\dim_B \ov{C}=\dim_T \ov{C}=d$ by Proposition \ref{C_2}\eqref{C_2(1)}, $\dim K <d.$ Suppose $K\neq 0$ and  $\q \in \Ass_B (K).$ Then as $\q \in \Ass_B (B(-m)),$ $\dim ~(B/\q)=d.$  This implies that $\dim_B K =d,$ which is a contradiction. Hence $K=(0)$ and thus $\ov{C} \simeq B(-m).$ Therefore $\ell_{T_\p} (\ov{C}_\p)=\ell_{B_\p} (B(-m)_\p)=1.$ 
Now the assertion follows from Proposition \ref{C_2}\eqref{C_2(3)}. 

\vskip 2mm
$(i):$ Follows from the assertion (3). 

$(ii)$ and $(iii):$ By Proposition \ref{C_2}\eqref{C_2(4)} we have
\begin{eqnarray*}
{HS}_{\overline{G}(I)}(z)=\frac{\ell_R(R/\ov{I})+(\e_0(I)-\ell_R(R/\ov{I})-\ell_R(\overline{I^2}/J\overline{I}))z+\ell_R(\overline{I^2}/J\overline{I})z^2}{(1-z)^d}-(1-z)HS_{\ov{C}}(z).
\end{eqnarray*}
Since $\ov{C} \simeq B(-m),$
$$
HS_{\ov{C}}(z)=\frac{z^m}{(1-z)^d}.
$$
Thus the assertions $(ii)$ and $(iii)$ follow.

$(iv):$
Since $\ov{C} \simeq B(-m),$ $\ov{C}$ is a Cohen-Macaulay $T$-module.
Therefore by Proposition \ref{C_2}\eqref{C_2(5)} $\depth ~\ov{G}(I) \geq d-1$.

$(v):$
Let $\ov{G}(I)$ be Cohen-Macaulay. Suppose that $\ov{I^3} \subseteq J.$ 
Then $\ov{I^n} \subseteq J$ for all $n \geq 3.$ 
Hence thanks to the Valabrega-Valla criterion $($c.f. \cite[Theorem 1.1]{RV10}$)$, $\ov{I^{n+1}}=J \cap \ov{I^{n+1}}=J \ov{I^n}$ for all $n \geq 2.$ 
This implies that $\ov{\rm{r}}_J(I) \leq 2$ and hence by Lemma \ref{fact1}\eqref{fact(2)} $\ov{C}=0$ which is a contradiction. Therefore $\ov{I^3} \nsubseteq J .$ 

Let $\ov{I^3} \nsubseteq J. $ 
Then, since $J \ov{I^2} \subseteq J \cap \ov{I^3} \subsetneq \ov{I^3},$ we get that $\ov{C}_2 = \ov{I^3}/J \ov{I^2} \neq 0$ so that $m=2.$ Since $\ell_R(\ov{I^3}/J \ov{I^2})=1$ by assertion $(3)$ and $J \ov{I^2} \subseteq J \cap \ov{I^3} \subsetneq \ov{I^3},$ $J \ov{I^2} = J \cap \ov{I^3}.$ 
Because $\ov{\rm{r}}_J(I)=3$ by assertion $(i),$ $\ov{I^{n+1}}=J \ov{I^n} $ for all $n \geq 3.$ 
Therefore we have $J \cap \ov{I^{n+1}}= J \ov{I^n}$ for all $n \geq 0$.
Hence thanks to the Valabrega-Valla's criterion, $\ov{G}(I)$ is Cohen-Macaulay.
\end{proof}

The above theorem extends and reproves   \cite[Theorem 2.5]{CPR16} because $\ell_R(\ov{I^2}/J \ov{I})$ can be any integer.

\begin{cor}
 Suppose the equality $\ov{e}_1(I)=\ov{e}_0(I)-\ell_R(R/\ov{I})+\ell_R(\ov{I^2}/J \ov{I})+1$ holds true. Then $\depth ~\ov{G}(I) \geq d-1$.
\end{cor}

\section{Applications and Examples}
In this section we present  some interesting consequences of Theorem \ref{maintheorem}. Next result specifies Theorem \ref{maintheorem} when the ideal $I$ is normal, that is $ \ov{I^n} =I^n $ for all $n \ge 0.$

\begin{cor}\label{normal}
Suppose that $I$ is normal.
Then the following three conditions are equivalent to each others where $\e_i(I)=\e_i(\{I^n\}_{n \in \mathbb{Z}})$ denotes the i-th Hilbert coefficient of an $I$-adic filtration:
\begin{enumerate}
 \item  $\e_1(I)=\e_0(I)-\ell_R(R/I)+\ell_R(I^2/JI)+1;$
 \item $\overline{C} \cong B(-2)$ as graded $T$-modules;
 \item $\ell_R(I^{3}/JI^2)=1$ and $I^4=JI^3$.
\end{enumerate}
When this is the case, the following assertions also follow:
\begin{itemize}
\item[(1)] $ HS_{G(I)}(z)=\frac{\ell_R(R/I)+(\e_0(I)-\ell_R(R/I)-\ell_R(I^2/JI))z+(\ell_R(I^2/JI)-1)z^2+z^{3}}{(1-z)^d}.$
\item[(2)] $\e_2(I)=\ell_R(I^2/JI)+2$ if $d \geq 2$, $\e_3(I)=1$ if $d \geq 3$, and $\e_i(I)=0$ for $4 \leq i \leq d.$
\item[(3)] $\depth ~G(I) \geq d-1.$
\item[(4)] $G(I)$ is Cohen-Macaulay if and only if $I^3 \nsubseteq J.$
\end{itemize}
\end{cor}
\begin{proof}
 Since $I$ is normal, $\ov{C}=(0)$ if and only if $\ov{C}_2=(0).$ 
By Theorem \ref{maintheorem} $\e_1(I)=\e_0(I)-\ell_R(R/I)+\ell_R(I^2/JI)+1$ implies that $\ov{C} \cong B(-m)$ for some $m \geq 2.$ 
Hence $\ov{C} \neq (0)$ which implies that $\ov{C}_2 \neq (0).$ Therefore $m \leq 2.$ 
Hence $m=2.$ Now the result follows from Theorem \ref{maintheorem}.
 \end{proof}

The following example, due to Huckaba and Huneke \cite[Theorem 3.12]{HH99}, illustrates that if $I$ is normal and $\e_1(I)=\e_0(I)-\ell_R(R/I)+\ell_R(I^2/J I)+1,$ then $G(I)$ need not be Cohen-Macaulay. 

\begin{ex}\label{example}
 Let $K$ be a field of characteristic $\not= 3$ and set $R =
K[\![X,Y,Z]\!]$, where $X,Y,Z$ are indeterminates. Let $N = (X^4,
X(Y^3+Z^3), Y(Y^3+Z^3), Z(Y^3+Z^3))$ and set $I = N + {\mathfrak
m}^5$, where ${\mathfrak m}$ is the maximal ideal of $R$. The
ideal $I$ is a normal ${\mathfrak m}$-primary ideal whose
associated graded  ring ${G}(I)$ has depth $d-1$, where
$d(=3)$ is the dimension of $R$. Moreover,
\[
HS_{{G}(I)}(t)=\frac{31 + 43t +t^2 + t^3}{(1-t)^3},
\]
and hence $\ell_R(R/I)=31,~\e_0(I)=76,~ \e_1(I)=48,$ $\ell_R(I^2/JI)=2.$ Thus $\e_1(I)=\e_0(I)-\ell_R(R/I)+\ell_R(I^2/J I)+1.$
\end{ex}

It is well-known that $\ov{\e}_1(I) = \ov{\e}_0(I)-\ell_R(R/\ov{I})$ implies that $\ov{\rm{r}}_J(I) \leq 1$ for any minimal reduction $J$ of $I$ and hence $\ov{G}(I)$ is Cohen-Macaulay (c.f.\cite[Corollaries 4.8 and 4.9]{HM97}). In \cite[Theorem 2.5]{CPR16} authors proved that if $\ov{\e}_1(I) = \ov{\e}_0(I)-\ell_R(R/\ov{I})+1,$ then $\depth ~\ov{G}(I) \geq d-1.$ In the following corollary we extend their result by computing the depth of $\ov{G}(I)$ explicitly.

\begin{cor}\label{CPR}
Let $(R,\m)$ be an analytically unramified Cohen-Macaulay local ring of dimension $d>0$ and $I$ an $\m$-primary ideal in $R.$  
 Suppose that $\ov{\e}_1(I) = \ov{\e}_0(I)-\ell_R(R/\ov{I})+1,$ then $\ell_R(\ov{I^2}/J \ov{I}) \leq 1$ and 
 \[ 
 \depth~\ov{G}(I) = \begin{cases}
                     d & \mbox{ if } \ell_R(\ov{I^2}/J \ov{I}) =1 \\
                     d-1 & \mbox{ if } \ell_R(\ov{I^2}/J \ov{I}) =0.
                    \end{cases}
\] 
\end{cor}
\begin{proof}
Let $J \subseteq I$ be a minimal reduction of $I.$ By Corollary \ref{EV}  $\ell_R(\ov{I^2}/J \ov{I}) \leq 1.$ 
If $\ell_R(\ov{I^2}/J \ov{I}) =1 ,$ then by Corollary \ref{EV} $\ov{\rm{r}}_J(I) \leq 2$ and $\ov{G}(I)$ is Cohen-Macaulay. 

Suppose $\ell_R(\ov{I^2}/J \ov{I}) =0,$ that is $\ov{I^2}=J \ov{I}.$ Then $\ov{I^3} \subseteq \ov{I^2} \subseteq J.$ Hence by Theorem \ref{maintheorem} $\depth~\ov{G}(I) = d-1.$ \qedhere
 \end{proof}

In the following theorem we analyze the case $\ov{e}_2(I) \leq \ell_R(\ov{I^2}/J \ov{I})+ 2$ (in particular, $\ov{e}_2(I) \leq 2$).

\begin{cor}\label{e2}
 Let $(R,\m)$ be an analytically unramified Cohen-Macaulay local ring of dimension $d>0$ and $I$ an $\m$-primary ideal in $R.$ If $\ov{e}_2(I) \leq \ell_R(\ov{I^2}/J \ov{I})+ 2,$ then $\depth~\ov{G}(I) \geq d-1.$  
\end{cor}
\begin{proof}
By \cite[Theorem 2]{Itoh92}
\begin{equation}
\label{Eqn:e2inequality}
\ov{e}_2(I) \geq \ov{e}_1(I) -\ov{e}_0(I) + \ell_R (R/\ov{I}) \geq \ell_R(\ov{I^2}/J \ov{I})
\end{equation}
and either of the equality in \eqref{Eqn:e2inequality} implies that $\ov{I^{n+2}}=J^n \ov{I^2}$ for all $n \geq 0$ (thus the other equality in \eqref{Eqn:e2inequality}) and hence $\ov{G}(I)$ is Cohen-Macaulay by \cite[Proposition 3]{Itoh88}. Therefore it is enough to consider the case $\ov{e}_2(I)=\ell_R(\ov{I^2}/J \ov{I})+2.$ In this case, each inequality in \eqref{Eqn:e2inequality} is strict. Therefore $\ov{e}_1(I)=\ov{e}_0(I) + \ell_R (R/\ov{I})+\ell_R(\ov{I^2}/J \ov{I})+1$ holds true. Hence by Theorem \ref{maintheorem} $\depth~\ov{G}(I) \geq d-1.$ 
\end{proof}


\end{document}